\documentclass[11pt,a4paper]{article}
\usepackage{amsmath,amsfonts,amsxtra,amssymb,amscd,amsthm}
\usepackage{indentfirst,anysize,graphicx}
\marginsize{3cm}{2cm}{1cm}{1cm}
\newtheorem{theorem}{Theorem}[section]

\newtheorem{proposition}{Proposition}[section]
\newtheorem{lemma}{Lemma}[section]
\newtheorem{definition}{Definition}[section]
\newtheorem{remark}{Remark}[section]

\begin{document}

\title{On the Rotar central limit theorem for sums of a random number of independent random variables}
\author{Tran Loc Hung\footnote{Ho Chi Minh City, Vietnam. Email: tlhungvn@gmail.com}}
\maketitle

\begin{abstract}
The Rotar central limit theorem is a remarkable theorem in the non-classical version since it does not use the condition of  asymptotic infinitesimality for the independent individual summands, unlike the theorems named Lindeberg's and Lindeberg-Feller's in the classical version. The Rotar central limit theorem   generalizes the classical Lindeberg-Feller central limit theorem since the Rotar condition is weaker than Lindeberg's. 

The main aim of this paper is to introduce the Rotar central limit theorem for sums of a random number of independent (not necessarily identically distributed) random variables and the conditions for its validity. The order of approximation in this theorem is also considered in this paper.
\end{abstract}


\section{Introduction and preliminary information}\label{sec:1}
 Let $(X_{j}, j\geq 1)$ be a sequence of independent (not necessarily identically distributed) random variables defined on a probability space $(\Omega, \mathcal{A}, \mathbb{P}),$ and suppose that $ \mathbb{E}X_{j}=0, \quad \mathbb{D}X_{j}=\sigma^{2}_{j}\in (0,+\infty)\quad\text{for}\quad j=1, 2, \cdots .$  Let us put
\[
S_{n}=\sum\limits_{j=1}^{n}X_{j}\quad\text{and}\quad B^{2}_{n}=\mathbb{D}S_{n}=\sum\limits_{j=1}^{n}\sigma^{2}_{j}.
\] 
From now on, we suggest that, for $n\geq 1, \quad 0 < B_{n}=\sqrt{\mathbb{D}S_{n}}$ and $B_{n}\longrightarrow\infty$ when $n\to\infty.$\\
It is clear that, for $n\geq 1,$ the desired  sums  $\bigg(\frac{S_{n}-\mathbb{E}S_{n}}{\sqrt{\mathbb{D}S_{n}}} \bigg)$ have the following moments:
\[
\mathbb{E}\bigg(\frac{S_{n}-\mathbb{E}S_{n}}{\sqrt{\mathbb{D}S_{n}}} \bigg)=\mathbb{E}\bigg(\frac{S_{n}}{B_{n}} \bigg)=0\quad\text{and}\quad \mathbb{D}\bigg(\frac{S_{n}-\mathbb{E}S_{n}}{\sqrt{\mathbb{D}S_{n}}} \bigg)=\mathbb{D}\bigg(\frac{S_{n}}{B_{n}} \bigg)=1.
\]
This allows us to think about the closeness of the distribution of the normalized sum $\frac{S_{n}}{B_{n}}$ to the standard normal distribution function $\Phi_{0,1}(x)$ when $n\to\infty$ in the sense of the central limit theorem (CLT).

Write
\[
\Delta_{n}:=\sup\limits_{x\in\mathbb{R}}\bigg|P\bigg(\frac{S_{n}}{B_{n}}<x\bigg)-\Phi_{0,1}(x) \bigg|.
\] 
\begin{definition}\label{def:1.1}
A sequence $(X_{j}, j\geq 1)$ is said to satisfy the CLT, if 
 \begin{equation}\label{equ:1.1}
\lim\limits_{n\to\infty}\Delta_{n}=0 
 \end{equation}
holds.
 \end{definition}
 In problems of checking the validity of the CLTs in (\ref{equ:1.1}),  the following conditions play an important role:
\begin{enumerate}
\item  The Lyapunov condition:   for any $\delta\in (0, 1]$ 
\begin{equation}\label{equ:1.2}
 \lim\limits_{n\to\infty}\frac{1}{B^{2+\delta}_{n}}\sum\limits_{j=1}^{n}\int\limits_{-\infty}^{+\infty}|x|^{2+\delta}dF_{j}(x)=0.
\end{equation}
\item The Lindeberg condition: for each $\epsilon >0,$
\begin{equation}\label{equ:1.3}
\lim\limits_{n\to\infty}\frac{1}{B^{2}_{n}}\sum\limits_{j=1}^{n}\int\limits_{|x|>\epsilon B_{n}}x^{2}dF_{j}(x)=0.
\end{equation}
\item The Feller condition: 
\begin{equation}\label{equ:1.4}
\lim\limits_{n\to\infty}\frac{1}{B^{2}_{n}}\max\limits_{1\leq j\leq n}\sigma^{2}_{j}=0.
\end{equation}
\end{enumerate}
 It is easy to check the following implication:
 \begin{proposition}\label{pro:1.1}\quad For every $\epsilon>0$ and for $ \delta\in (0, 1]$ 
 \[
 (\ref{equ:1.2}) \Longrightarrow (\ref{equ:1.3}) \Longrightarrow (\ref{equ:1.4}).
 \]
 
\end{proposition}
The following theorem states that (\ref{equ:1.3}) is a sufficient condition for the validity of  CLT in (\ref{equ:1.1}). The detailed proofs  can be found in \cite{Shiryaev1996} (Chapter III, Theorem 1 on  page 329).
\begin{theorem}\label{the:1.1} (Lindeberg theorem)\quad For every $\epsilon>0$ 
\[
(\ref{equ:1.3})\Longrightarrow  (\ref{equ:1.1}).
\]
\end{theorem}
Using the Proposition \ref{pro:1.1}, the Lyapunov theorem confirms that the CLT is also valid if  the Lyapunov condition (\ref{equ:1.2}) holds.
\begin{theorem}\label{the:1.2} (Lyapunov theorem)\quad For $\delta\in (0, 1]$ 
\[
(\ref{equ:1.2})\Longrightarrow  (\ref{equ:1.1}).
\]
\end{theorem}
 Next, the Lindeberg-Feller theorem improves the Lindeberg theorem  as  the following result.
 \begin{theorem}\label{the:1.3} (Lindeberg-Feller theorem)\quad For every $\epsilon>0$ 
\begin{enumerate}
\item[{i}.] The CLT is valid if  the Lindeberg condition (\ref{equ:1.3}) holds., i.e.,
\[
(\ref{equ:1.3})\Longrightarrow (\ref{equ:1.1}).
\]
\item[{ii}.] Under the Feller condition (\ref{equ:1.4}), the CLT follows the (\ref{equ:1.3}), that is,
\[
(\ref{equ:1.1}) \quad\&\quad (\ref{equ:1.4}) \Longrightarrow  (\ref{equ:1.3}).
\]
\end{enumerate}
\end{theorem}
The references \cite{Feller1971}, \cite{Gnedenko1949}, \cite{Petrov1995}, \cite{Renyi1970}, \cite{Rotar1996}, and \cite{Shiryaev1996} are devoted to the study of the CLTs for sums of independent random variables and the Lindeberg condition (\ref{equ:1.3}) with its closed conditions. 

The well-known CLTs like Lyapunov's, Lindeberg's, or Lindeberg-Feller's for the sequence $(X_{j}, j\geq 1)$ are so-called  classical since they use  the following condition:
 \begin{equation}\label{equ:1.5}
 \lim\limits_{n\to\infty}P(\sup\limits_{1\leq j\leq n} |X_{j}|>\epsilon B_{n})=0,  
\end{equation}  
 for ever $\epsilon>0.$ The condition (\ref{equ:1.5}) is called condition of asymptotic infinitesimality for the independent individual summand $\frac{X_{j}}{B_{n}}$ in sums $\frac{S_{n}}{B_{n}}$ for $j=1, 2, \cdots$ and $n\geq 1.$ By Chebyshev's inequality (see for instance \cite{Petrov1995}, page 124 ), for every $\epsilon>0,$ it is clear to check that
 \[
 (\ref{equ:1.4})\Longrightarrow (\ref{equ:1.5}). 
 \]
Following Zolotarev \cite{Zolotarev1997}, the CLTs that do not use the condition (\ref{equ:1.5}) are called non-classical. An example of a CLT in the non-classical version is introduced by Shiryaev (\cite{Shiryaev1996}, page 337). This remarkable example shows that there are nondegenerate random variables such that neither the Lindeberg condition (\ref{equ:1.3}) and the Feller condition (\ref{equ:1.4})   nor the condition of asymptotic infinitesimality (\ref{equ:1.5}) are satisfied, but nevertheless the "non-classical" CLT is valid. 

The following result is a non-classical CLT that is modified from the Rotar theorem (see for instance  \cite{Shiryaev1996},  Theorem 1, page 338) as follows:
\begin{theorem}\label{the:1.4}\quad To have  (\ref{equ:1.1}), it is sufficient (and necessary) that for every $\epsilon>0$ the condition 
\begin{equation}\label{equ:1.6}
\lim\limits_{n\to\infty} \frac{1}{B_{n}}\sum\limits_{j=1}^{n}\int\limits_{|x|>\epsilon B_{n}}|x|\cdot |F_{j}(x)-\Phi_{j}(x)|dx=0,  
\end{equation}  
is satisfied, where $F_{j}(x)=P(X_{j}<x)$ and $\Phi_{j}(x)=\Phi_{0,1}(x/\sigma_{j}),$ for $j=1, 2, \cdots, n.$
\end{theorem}
 The condition (\ref{equ:1.6}) is named Rotar's. Note that the Rotar condition is  weaker than the following Lindeberg condition (\ref{equ:1.3})  in the sense that $(\ref{equ:1.3})\Longrightarrow (\ref{equ:1.6})$ (the detailed proof can be found in \cite{Shiryaev1996}, Ch. 5, 18.3, p. 310-311). Moreover, the Rotar condition (\ref{equ:1.6}) has no relation to the condition of  asymptotic infinitesimality (\ref{equ:1.5}). The Rotar theorem \ref{the:1.4} confirms that the Rotar condition (\ref{equ:1.6}) is sufficient and necessary for the validity of the CLT for  sums of independent random variables. Thus, the Rotar CLT generalizes the classical Lindeberg-Feller CLT.  Recently, the extension and generalization of the Rotar condition and the Rotar CLT have been studied by in \cite{Formanov2002},  \cite{Formanov2010}, \cite{Formanov2019}, and \cite{Formanov2022}.

Following \cite{Robbins1948}, many authors (e. g.  \cite{Blum1963}, \cite{Feller1971}, \cite{Gnedenko1996}, \cite{Gut2005}, \cite{Kalashnikov1997}, \cite{Kruglov1990}, \cite{Neammanee2004}, \cite{Renyi1970}, \cite{Robbins1948}, \cite{Rychlik1976}, \cite{Rychlik1979}, \cite{Shanthikumar1984}) have established the  CLTs for random sums with the  conditions  for their validity.  Therefore, randomizing the theorem \ref{the:1.4} with its order of approximation is an interesting research issue.

The main purpose of this paper is to extend Theorem \ref{the:1.4} to the case of random summations in the sense that the non-random number $n$ of the normalized sums $\frac{S_{n}}{B_{n}}$ will be replaced by the positive integer-valued random variables $\nu_{n}$ which are independent of all $X_{j}$ for $j\geq 1$ and $n\geq 1.$ The Rotar CLT for sums of a random number of independent (not necessarily identically distributed) random variables is established with the random Rotar condition. The order of approximation in that random Rotar CLT is also established in this paper. 

This paper is organized as follows: In Sec. 2, we give preliminary results of the random conditions for the validity of the CLTs for random sums. The random Rotar CLT with its detailed proof is presented in Sec. 3, along with the order of approximation.  

\section{Preliminary results}\label{sec:2}
 Let $(\nu_{n}, n\geq 1)$ be a sequence of positive integer-valued random variables such that for each $n,$ the  $\nu_{n}$ and random variables $X_{1}, X_{2}, \cdots$ are independent. Assume that 
\begin{equation}\label{equ:2.1}
 \nu_{n}\stackrel{P}{\longrightarrow}+\infty, \quad\text{as}\quad n\longrightarrow+\infty,
\end{equation}
 where $\stackrel{P}{\longrightarrow}$ denotes the convergence in probability.
 Let us set the following random sums:
 \[
 S_{\nu_{n}}=\sum\limits_{j=1}^{\nu_{n}}X_{j}\quad\text{and}\quad B^{2}_{\nu_{n}}=\sum\limits_{j=1}^{\nu_{n}}\sigma^{2}_{j}.
  \]
  Throughout this paper, for $n\geq 1,$ assume that the random variables $B_{\nu_{n}}>0$ almost surely and $B_{\nu_{n}}\stackrel{P}{\longrightarrow}\infty$ when $n\to\infty.$ \\
  Following \cite{Gut2005}, it is easy to check that
  \[
\begin{split}
 & \mathbb{E}S_{\nu_{n}}=\sum\limits_{n=1}^{\infty}P(\nu_{n}=n)\mathbb{E}S_{n}=\sum\limits_{n=1}^{\infty}P(\nu_{n}=n)\sum\limits_{j=1}^{n}\mathbb{E}X_{j}=0,\\
& \mathbb{D}S_{\nu_{n}}=\sum\limits_{n=1}^{\infty}P(\nu_{n}=n)\mathbb{D}S_{n}=\sum\limits_{n=1}^{\infty}P(\nu_{n}=n)\sum\limits_{j=1}^{n}\mathbb{D}X_{j}=B^{2}_{\nu_{n}}.
\end{split}  
  \]
 Consequently,
  \[
  \mathbb{E}\bigg(\frac{S_{\nu_{n}}-\mathbb{E}S_{\nu_{n}}}{\sqrt{\mathbb{D}S_{\nu_{n}}}} \bigg)= \mathbb{E}\bigg(\frac{S_{\nu_{n}}}{B_{\nu_{n}}} \bigg)=0\quad\text{and}\quad \mathbb{D}\bigg(\frac{S_{\nu_{n}}-\mathbb{E}S_{\nu_{n}}}{\sqrt{\mathbb{D}S_{\nu_{n}}}} \bigg)=\mathbb{D}\bigg(\frac{S_{\nu_{n}}}{B_{\nu_{n}}} \bigg)=1.
  \]
Now we are interested in the CLT for random summations presented in the following form:
\[
\lim\limits_{n\to\infty}P\bigg(\frac{1}{B_{\nu_{n}}}\sum\limits_{j=1}^{\nu_{n}}X_{j}<x \bigg)=\Phi_{0,1}(x), \quad\text{for}\quad x\in\mathbb{R}.
\]
Write
\[
\Delta_{\nu_{n}}=\sup\limits_{x\in\mathbb{R}}\bigg|P\bigg(\frac{1}{B_{\nu_{n}}}\sum\limits_{j=1}^{\nu_{n}}X_{j}<x\bigg)-\Phi_{0,1}(x) \bigg|.
\]
Then, the CLT for random sums states as follows:
\begin{definition}\label{def:2.1} (Random CLT)\quad A sequence $(X_{j}, j\geq 1)$ is said to obey the random CLT  if 
\begin{equation}\label{equ:2.2}
\lim\limits_{n\to\infty}\Delta_{\nu_{n}}=0.
\end{equation}
\end{definition}
\begin{remark}\label{rem:2.1}\quad If $P(\nu_{n}=n)=1,$ the (8) deduces the (1). 
\end{remark}
It is obvious that the CLTs for random sums are extensions and generalizations  of the CLTs for non-random sums, and the typical results may be found in \cite{Feller1971}, \cite{Gnedenko1996}, \cite{Gut2005}, \cite{Kalashnikov1997}, \cite{Kruglov1990}, \cite{Neammanee2004}, \cite{Renyi1970}, \cite{Robbins1948}, and \cite{Rychlik1976}, etc.

The following are conditions in random versions for the validity of CLTs for random summations of the independent random variables.
\begin{definition}\label{def:2.2}(Random Lindeberg condition)\quad A sequence $(X_{j}, j\geq 1)$ is said to satisfy the random Lindeberg condition  if,  for any $\epsilon >0$ 
\begin{equation}\label{equ:2.3}
\lim\limits_{n\to\infty}\mathbb{E}\bigg\{\frac{1}{B^{2}_{\nu_{n}}}\sum\limits_{j=1}^{\nu_{n}}\int\limits_{|x|>\epsilon B_{\nu_{n}}}x^{2}dF_{j}(x) \bigg\}=0.
\end{equation} 
\end{definition}
\begin{remark}\label{rem:2.2}\quad According to \cite{Gut2005} (Section 9, page 192), when $P(\nu_{n}=n)=1$ for every $n\geq 1,$ the random Lindeberg condition  (\ref{equ:2.3}) deduces the non-random Lindeberg condition (\ref{equ:1.3}) by the following fact:
\[
\mathbb{E}\bigg\{\frac{1}{B^{2}_{\nu_{n}}}\sum\limits_{j=1}^{\nu_{n}}\int\limits_{|x|>\epsilon B_{\nu_{n}}}x^{2}dF_{j}(x) \bigg\}
=\sum\limits_{n=1}^{\infty}P(\nu_{n}=n)\mathbb{E}\bigg\{\frac{1}{B^{2}_{n}}\sum\limits_{j=1}^{n}\int\limits_{|x|>\epsilon B_{n}}x^{2}dF_{j}(x)\bigg\},
\]
for any $\epsilon>0$ and for $n\geq 1.$
\end{remark}
The following result was proved by Rychlik \cite{Rychlik1976} (Lemma 1, on page 149).
\begin{proposition}\label{pro:2.1}\quad If a sequence $(X_{j}, j\geq 1)$ of independent random variables satisfies the non-random Lindeberg condition  (\ref{equ:1.3}) and if the condition (\ref{equ:2.1}) holds,  then the  random Lindeberg condition (\ref{equ:2.3}) is valid. 
\end{proposition}
\begin{proposition}\label{pro:2.2}\quad The random Lindeberg condition (\ref{equ:2.3}) is satisfied automatically for any sequence of independent identically distributed (i.i.d.) random  variables.
\end{proposition} 
\begin{proof} \quad Let   $(X_{j}, j\geq 1)$ be a sequence of i.i.d. random variables such that
\[
\mathbb{E}X_{j}=0, \quad \mathbb{D}X_{j}=\sigma^{2}\in (0, \infty), \quad B^{2}_{n}=\sum\limits_{j=1}^{n}\mathbb{D}X_{j}=n\sigma^{2}, \quad\text{for}\quad  j=1, 2, \cdots, n; n\geq 1.
\]
Then, for any $\epsilon>0,$ we have
\[
\frac{1}{B^{2}_{n}}\sum\limits_{j=1}^{n}\int\limits_{|x|>\epsilon B_{k}}x^{2}dF_{j}(x)=\frac{n}{n\sigma^{2}}\int\limits_{|x|>\epsilon \sigma\sqrt{n}}x^{2}dF_{j}(x)=\frac{1}{\sigma^{2}}\int\limits_{|x|>\epsilon \sigma\sqrt{n}}x^{2}dF_{1}(x).
\]
Following \cite{Gut2005}, for every $\epsilon>0,$ it follows that
\begin{equation}\label{equ:2.4}
\mathbb{E}\bigg\{\frac{1}{B^{2}_{\nu_{n}}}\sum\limits_{j=1}^{\nu_{n}}\int\limits_{|x|>\epsilon B_{\nu_{n}}}x^{2}dF_{j}(x)\bigg\}=\mathbb{E}\bigg\{\frac{1}{\sigma^{2}}\int\limits_{|x|>\epsilon \sigma \sqrt{\nu_{n}}}x^{2}dF_{1}(x)\bigg\}.
\end{equation}
The right side of (\ref{equ:2.4}) will vanish since $\{x: |x|>\epsilon \sigma \sqrt{\nu_{n}}\}\downarrow \emptyset$ with the suggestion that the condition (\ref{equ:2.1}) holds and  the boundness of  $\sigma^{2}=\mathbb{D}X_{1}<+\infty$ as suggested. Therefore, the random Lindeberg condition (\ref{equ:2.3}) is satisfied. 
\end{proof}
The following condition may be obtained from the random Lindeberg condition.
\begin{definition}\label{def:2.3}(Random Feller condition)\quad A sequence $(X_{n}, n\geq 1)$  is said to satisfy the random Feller condition if 
\begin{equation}\label{equ:2.5}
\lim\limits_{n\to\infty}\mathbb{E}\bigg\{\frac{1}{B^{2}_{\nu_{n}}}\max\limits_{1\leq j\leq \nu_{n}}\sigma^{2}_{j} \bigg\}=0.
\end{equation} 
\end{definition}
 \begin{proposition}\label{pro:2.3}\quad For every $\epsilon>0$ 
  \[
 (\ref{equ:2.3})\Longrightarrow (\ref{equ:2.5}).
 \] 
\end{proposition}
\begin{proof}
For any $\epsilon>0,$ it is easily seen that
\[
\begin{split}
\max\limits_{1\leq j\leq n}\sigma^{2}_{j}&\leq \sum\limits_{j=1}^{n}\sigma^{2}_{j}=\sum\limits_{j=1}^{n}\int\limits_{|x|\leq\epsilon B_{n}}x^{2}dF_{j}(x)+\sum\limits_{j=1}^{n}\int\limits_{|x| >\epsilon B_{n}}x^{2}dF_{j}(x)\\
&\leq \epsilon^{2}\sum\limits_{j=1}^{n}\sigma^{2}_{j}+\sum\limits_{j=1}^{n}\int\limits_{|x|>\epsilon B_{n}}x^{2}dF_{j}(x).
\end{split}
\]
Therefore,  using the equality $B^{2}_{n}=\sum\limits_{j=1}^{n}\sigma^{2}_{j},$ for any $\epsilon >0,$ we conclude that
\[
\begin{split}
\frac{1}{B^{2}_{n}}\max\limits_{1\leq j\leq n}\sigma^{2}_{j}\leq \epsilon^{2}+\frac{1}{B^{2}_{n}}\sum\limits_{j=1}^{n}\int\limits_{|x|>\epsilon B_{n}}x^{2}dF_{j}(x).
\end{split}
\]
Then, according to \cite{Gut2005}, for any $\epsilon >0,$ it follows that
\[
\mathbb{E}\bigg\{\frac{1}{B^{2}_{\nu_{n}}}\max\limits_{1\leq j\leq \nu_{n}}\sigma^{2}_{j}\bigg\}\leq \epsilon^{2}+\mathbb{E}\bigg\{\frac{1}{B^{2}_{\nu_{n}}}\sum\limits_{j=1}^{\nu_{n}}\int\limits_{|x|>\epsilon B_{n}}x^{2}dF_{j}(x)\bigg\}.
\]
Applying (\ref{equ:2.3}), since $\epsilon >0$ can be chosen arbitrarily small, the proof is finished.
\end{proof}
\begin{remark}\label{rem:2.3}\quad  In the case $P(\nu_{n}=n)=1,$  the non-random Feller condition (\ref{equ:1.4}) is followed  by the random Feller condition (\ref{equ:2.5}).
\end{remark}
To compare with the main object of this paper, which is named Random Rotar CLT, one result is introduced below. Its detailed proof may be found in \cite{Hung2023}. 
\begin{theorem}\label{the:2.1}(Random Lindeberg-Feller CLT) \quad The following statement is true: 
\begin{enumerate}
\item[{i.}] A sequence $(X_{j}, j\geq 1)$ is satisfied  the random CLT, if (\ref{equ:2.3}) holds, that is,
\[
(\ref{equ:2.3})\Longrightarrow (\ref{equ:2.2})\quad\text{as}\quad n\to\infty, \quad\text{for every}\quad \epsilon>0.  
\]
\item[{ii.}] Under the random Feller condition (\ref{equ:2.5}), the random CLT  deduces the (\ref{equ:2.3}), i.e., 
\[
(\ref{equ:2.5})\quad\&\quad  (\ref{equ:2.2})\Longrightarrow (\ref{equ:2.3})\quad\text{as}\quad n\to\infty, \quad\text{for every}\quad \epsilon>0.  
\]
\end{enumerate}
\end{theorem}
To establish the main results of this paper, we need an  expression of the standard normal distributed random variable $X^{*}\stackrel{D}{\sim}\mathcal{N}(0,1)$ with zero mean, unit variance and characteristics function $f_{X^{*}}(t)=e^{-\frac{1}{2}t^{2}}.$ Based on $X^{*}$ and $0<\sigma_{j}=\sqrt{\mathbb{D}X_{j}}$ for $j=1, 2, \cdots,$ the sequence $(X_{j}, j\geq 1)$ is formed by $X^{*}_{j}\stackrel{D}{=}\sigma_{j}X^{*}$ for $j=1, 2, \cdots.$ Assume that the random variables $X^{*}_{j}, X^{*}_{2}, \cdots$ are independent, and  three sequences $(X_{j}, j\geq 1), (X^{*}_{j}, j\geq 1), (\nu_{n}, n\geq 1)$ are independent for $j=1, 2, \cdots$ and $n\geq 1.$ Thus, $(X^{*}_{j}, j\geq 1) $ is a sequence of  independent normal distributed random variables with zero means, finite variances $\sigma^{2}_{j}\in (0, +\infty)$ and characteristic function $f_{X^{*}_{j}}(t)=e^{-\frac{1}{2}\sigma^{2}_{j}t^{2}},$ for $j\geq 1.$ Here and from now on,  the distribution function of $X^{*}_{j}$ is denoted by
\[
P(X^{*}_{j}<x)=\Phi_{j}(x)=\Phi_{0,1}(x/\sigma_{j})\quad\text{for}\quad j=1, 2, \cdots.
\]
 \begin{lemma}\label{lem:2.1}\quad For the $X^{*}\stackrel{D}{\sim}\mathcal{N}(0,1),$ the following expression 
 \[
 X^{*}\stackrel{D}{=}\frac{1}{B_{\nu_{n}}}\sum\limits_{j=1}^{\nu_{n}}X^{*}_{j}
 \]
 is true, where $\stackrel{D}{=}$ denotes the equality in distribution.
 \end{lemma}
 \begin{proof}\quad Following \cite{Gut2005}, in view of independence of the random variables $X^{*}_{j}\stackrel{D}{\sim}\mathcal{N}(0, \sigma^{2}_{j})$ for $j=1, 2, \cdots,$ the characteristic function of the random variable $\frac{1}{B_{\nu_{n}}}\sum\limits_{j=1}^{\nu_{n}}X^{*}_{j}$ is given as follows
 \begin{equation}\label{equ:2.6}
 \begin{split}
& f_{\frac{1}{B_{\nu_{n}}}\sum\limits_{j=1}^{\nu_{n}}X^{*}_{j}}(t)=
 \sum\limits_{n=1}^{\infty}P(\nu_{n}=n) f_{\frac{1}{B_{n}}\sum\limits_{j=1}^{n}X^{*}_{j}}(t)\\
 &=\sum\limits_{n=1}^{\infty}P(\nu_{n}=n)\prod_{j=1}^{n}f_{X^{*}_{j}}(t/B_{n})
 =\sum\limits_{n=1}^{\infty}P(\nu_{n}=n)e^{-\frac{1}{2}\frac{t^{2}}{B^{2}_{n}}\sum\limits_{j=1}^{n}\sigma^{2}_{j}}\\
 &=e^{-\frac{1}{2}t^{2}}=f_{X^{*}}(t).
\end{split} 
 \end{equation}
 In (\ref{equ:2.6}), we used the fact that $B^{2}_{n}=\sum\limits_{j=1}^{n}\sigma^{2}_{j}.$ Thus,  the characteristic function of $\frac{1}{B_{n}}\sum\limits_{j=1}^{n}X^{*}_{j}$ coincides with the characteristic function of standard normal distributed random variable $X^{*}.$ Following the uniqueness theorem (see for instance \cite{Gut2005}, Theorem 1.2, page 160), the proof  is complete.
 \end{proof}
 \begin{remark}\label{rem:2.4}\quad  When $P(\nu_{n}=n)=1,$ from Lemma \ref{lem:2.1}, it follows that
 \[
 X^{*}\stackrel{D}{=}\frac{1}{B_{n}}\sum\limits_{j=1}^{n}X^{*}_{j}.
 \]
\end{remark}
To extend the Rotar theorem \ref{the:1.4} to the case of  random summations, we need the random Rotar condition as follows.
\begin{definition}\label{def:2.4}\quad A sequence $(X_{j}, j\geq 1)$ is said to obey the random Rotar condition if, for any $\epsilon >0$
\begin{equation}\label{equ:2.7}
\lim\limits_{n\to\infty}\mathbb{E}\bigg\{\frac{1}{B_{\nu_{n}}}\sum\limits_{j=1}^{\nu_{n}}\int\limits_{|x|>\epsilon B_{\nu_{n}}}|x|\cdot |F_{j}(x)-\Phi_{j}(x)|dx\bigg\}=0,
\end{equation}
\end{definition}
The following result confirms that the random Rotar condition (\ref{equ:2.7}) is weaker than the random Lindeberg condition (\ref{equ:2.3}).
\begin{proposition}\label{pro:2.4}\quad For any $\epsilon >0,$ the random Rotar condition (\ref{equ:2.7}) is deduced by  the random Lindeberg condition  (\ref{equ:2.3}).
\end{proposition}
\begin{proof}\quad Following \cite{Rotar1996} (inequality on page 311), we have
\begin{equation}\label{equ:2.8}
\begin{split}
&\frac{1}{B_{n}}\sum\limits_{j=1}^{n}\int\limits_{|x|>\epsilon B_{n}}|x|\cdot |F_{j}(x)-\Phi_{j}(x)|dx
\leq \frac{1}{B^{2}_{n}}\sum\limits_{j=1}^{n}\int\limits_{|x|>\epsilon B_{n}}x^{2}dF_{j}(x)\\
&+\frac{1}{B^{2}_{n}}\sum\limits_{j=1}^{n}\sigma^{2}_{j}\int\limits_{|x|>\epsilon B_{n}/\sigma^{*}}x^{2}d\Phi_{0, 1}(x)=\frac{1}{B^{2}_{n}}\sum\limits_{j=1}^{n}\int\limits_{|x|>\epsilon B_{n}}x^{2}dF_{j}(x)+\int\limits_{|x|>\epsilon B_{n}/\sigma^{*}}x^{2}d\Phi_{0, 1}(x),
\end{split}
\end{equation}
where $\sigma^{*}=\max\limits_{1\leq j\leq n}\sigma_{j}$ and we used the fact that $B^{2}_{n}=\sum\limits_{j=1}^{n}\sigma^{2}_{j}.$ According to \cite{Gut2005},  from the inequality (\ref{equ:2.8}), we have
\begin{equation}\label{equ:2.9}
\begin{split}
&\mathbb{E}\bigg\{\frac{1}{B_{\nu_{n}}}\sum\limits_{j=1}^{\nu_{n}}\int\limits_{|x|>\epsilon B_{\nu_{n}}}|x|\cdot |F_{j}(x)-\Phi_{j}(x)|dx\bigg\}\leq \mathbb{E}\bigg\{\frac{1}{B^{2}_{\nu_{n}}}\sum\limits_{j=1}^{\nu_{n}}\int\limits_{|x|>\epsilon B_{\nu_{n}}}x^{2}dF_{j}(x)\bigg\}\\
&+\mathbb{E}\bigg\{\int\limits_{|x|>\epsilon B_{\nu_{n}}/\sigma^{*}}x^{2}d\Phi_{0,1}(x)\bigg\}.
\end{split}
\end{equation}
On account of  \cite{Rychlik1976} (Lemma 2, page 149), the random Lindeberg condition (\ref{equ:2.3}) takes place for both sequences $(X_{j}, j\geq 1)$ and $(X^{*}_{j}, j\geq 1).$ Two terms on the right side of (\ref{equ:2.9}) will vanish when $n$ tends to infinity. The proof is finished.
\end{proof}

\section{Random Rotar's CLT and  its order of approximation}\label{sec:3}
As a main result in this paper, the randomized Rotar's CLT for independent random variables is presented as follows. 
\begin{theorem}\label{the:3.1}(Random Rotar CLT)\quad For the every $\epsilon>0$ the  following statements are true:
\begin{enumerate}
\item[{i}.] A sequence $(X_{j}, j\geq 1)$ obeys the random Rotar CLT (\ref{equ:2.2}) if the random Rotar condition (\ref{equ:2.7}) holds.
\item[{ii}.] Under the random Feller condition (\ref{equ:2.6}), the (\ref{equ:2.2}) deduces the random Rotar condition (\ref{equ:2.7}).
\end{enumerate}
\end{theorem}
\begin{proof}
 ({i}).\quad Following the \cite{Shiryaev1996} (see inequality (6) on page 339), we have
 \begin{equation}\label{equ:3.1}
\bigg|\mathbb{E}(e^{itS_{\nu_{n}}/B_{\nu_{n}}})-e^{-t^{2}/2}\bigg|\leq \epsilon |t|^{3}+2t^{2}\mathbb{E}\bigg\{\frac{1}{B_{\nu_{n}}}\sum\limits_{j=1}^{\nu_{n}}\int\limits_{|x|>\epsilon B_{\nu_{n}}}|x|\cdot |F_{j}(x)-\Phi_{j}(x)|dx\bigg\},
\end{equation}
for every $\epsilon>0$ and $t\in\mathbb{R}.$
Applying the random Rotar condition  (\ref{equ:2.7}), by virtue of arbitrariness of $\epsilon,$ from above inequality  (\ref{equ:3.1}), it may be concluded that
\[
\lim\limits_{n\to\infty}\bigg|\mathbb{E}(e^{itS_{\nu_{n}}/B_{\nu_{n}}})-e^{-t^{2}/2}\bigg|=0.
\] 
According to the continuity theorem of characteristic functions (\cite{Gut2005}, Theorem 1.2, page 160), the above  limit expression shows that the (\ref{equ:2.2}) takes place for a sequence $(X_{j}, j\geq 1),$ i.e.,
\[
\lim\limits_{n\to\infty}\Delta_{\nu_{n}}=0.
\]
Thus, the first assertion ({i}) of the theorem is proved.\\
({ii}).\quad Using the fact that $B^{2}_{n}=\sum\limits_{j=1}^{n}\sigma^{2}_{j},$ it follows that
\[
\frac{1}{B^{4}_{n}}\sum\limits_{j=1}^{n}\sigma^{4}_{j}\leq \frac{1}{B^{2}_{n}}\times \sum\limits_{j=1}^{n}\sigma^{2}_{j}\times \frac{1}{B^{2}_{n}}\max\limits_{1\leq j\leq n}\sigma^{2}=
\frac{1}{B^{2}_{n}}\max\limits_{1\leq j\leq n}\sigma^{2}_{j}.
\]
Then, from  \cite{Formanov2019} (Theorem 1, point 3  on page 33),  with the "series form" $X_{n, j}=\frac{X_{j}}{B_{n}},$ for $n\geq 1, j=1, 2, \cdots, n$ and for any $\epsilon>0,$ we obtain
 \begin{equation}\label{equ:3.2}
 \begin{split}
&\frac{1}{B_{n}}\sum\limits_{j=1}^{n}\int\limits_{|x|>\epsilon B_{n}}|x|\cdot|F_{j}(x)-\Phi_{j}(x)|dx
\leq C(\epsilon)\bigg(\Delta_{n}+\frac{1}{B^{4}_{n}}\sum\limits_{j=1}^{n}\sigma^{4}_{j}\bigg)\\
&\leq C(\epsilon)\bigg(\Delta_{n}+\frac{1}{B^{2}_{n}}\max\limits_{1\leq j\leq n}\sigma^{2}_{j} \bigg),
\end{split}
\end{equation}
where $C(\epsilon)>0.$ \\
According to \cite{Gut2005}, from the above inequality (\ref{equ:3.2}), it follows that
 \begin{equation}\label{equ:3.3}
\mathbb{E}\bigg\{\frac{1}{B_{\nu_{n}}}\sum\limits_{j=1}^{\nu_{n}}\int\limits_{|x|>\epsilon B_{\nu_{n}}}|x|\cdot |F_{j}(x)-\Phi_{j}(x)|dx\bigg\}\leq C(\epsilon)\bigg[\Delta_{\nu_{n}}+\mathbb{E}\bigg(\frac{1}{B^{2}_{\nu_{n}}}\max\limits_{1\leq j\leq \nu_{n}}\sigma^{2}_{j}\bigg) \bigg],
\end{equation}
where $C(\epsilon)>0,$ for any $\epsilon>0.$ \\
According to (\ref{equ:3.3}), the random Rota condition (\ref{equ:2.7})  is satisfied if the (\ref{equ:2.2})  holds (i.e., $\Delta_{\nu_{n}}=o(1)$ when $n\to\infty$) under the random Feller condition (\ref{equ:2.5}). Therefore,  the following  implication 
\[
 (\ref{equ:2.5})\quad\&\quad (\ref{equ:2.2})\Longrightarrow (\ref{equ:2.8})
\]
is valid. The proof is complete.
\end{proof}
\begin{remark}\label{rem:3.1}\quad Theorem \ref{the:3.1} is an extension of Theorem \ref{the:1.4}, since in the case $P(\nu_{n}=1)=1,$ the first theorem is deduced by the second.
\end{remark}
\begin{remark}\label{rem:3.2}\quad It is clear that  the random Rotar CLT \ref{the:3.1} generalizes the random Lindeberg-Feller CLT \ref{the:2.1} since the random Rotar condition (\ref{equ:2.7}) is weaker than the Lindeberg condition (\ref{equ:2.3}).
\end{remark}
The following results devote the order of approximation in the Theorem \ref{the:3.1}  in terms of  "small-o" and "large-$\mathcal{O}$" estimates. Following \cite{Lindeberg1922} and according to so-called Lindeberg-Trotter method \cite{Trotter1959}, it follows that a sufficient condition for the validity of (\ref{equ:2.2}) is 
\[
\lim\limits_{n\to\infty}\bigg|\mathbb{E}f\bigg(\frac{S_{\nu_{n}}}{B_{\nu_{n}}}\bigg)- \mathbb{E}f(X^{*})\bigg| =0,
\]
for each  $f\in C^{2}_{B}(\mathbb{R})=\bigg\{f\in C_{B}(\mathbb{R}): f^{(j)})\in C_{B}(\mathbb{R}), j=1, 2\bigg\},$ where $C_{B}(\mathbb{R})$ denotes the class of bounded, uniformly continuous functions defined on the real axis $\mathbb{R}$ with a norm $||f||=\sup\limits_{x\in\mathbb{R}}|f(x)|.$ 

To establish the order of approximation in random Rotar CLT (\ref{the:3.1}), we need the concept of the modulus of continuity of function $f.$  Let $f\in C_{B}(\mathbb{R}.$ We denote $\omega(f; \epsilon)$ the modulus of continuity of $f,$ such that
\[
\omega(f; \epsilon)=\sup\limits_{|x-y|<\epsilon}|f(x)-f(y)|,
\]
for any $\epsilon>0.$ It is easy to check that
\begin{enumerate}
\item[{i.}] $\omega(f; \epsilon)$ is a  monotonely  decreasing function of $\epsilon$ with $\lim\limits_{\epsilon\to 0+}\omega(f; \epsilon)= 0,$  and
\item[{ii.}] For each $\lambda>0,\quad \omega(f; \lambda\epsilon)\leq (1+\lambda)\omega(f; \epsilon).$
\end{enumerate}
Furthermore, a function $f\in C_{B}(\mathbb{R})$ is said to satisfy a Lipschitz condition of order $\alpha, 0<\alpha\leq 1,$ in symbols $f\in Lip (\alpha, K)$ if $\omega(f; \epsilon)=K\epsilon^{\alpha},$ here $K\geq 0$ denotes a positive constant.
 \begin{theorem}\label{the:3.2}("Large-$\mathcal{O}$" approximation estimates)\quad Under the given assumptions for the sequences $(X_{j}, j\geq 1), (X^{*}_{j}, j\geq 1),$  and $(\nu_{n}, n\geq 1)$ for $j=1, 2, \cdots; n\geq 1,$  let us assume that 
\begin{equation}\label{equ:3.4}
  \sum\limits_{j=1}^{n}\bigg[\mathbb{E}|X_{j}|+\mathbb{E}|X^{*}_{j}|+2\sigma^{2}_{j}\bigg]\leq M_{1}<+\infty,
  \end{equation}
  for a  positive constant $M_{1}>0.$  \\
    Then, for any $f\in C^{1}_{B}(\mathbb{R})$
  \begin{equation}\label{equ:3.5}
 \bigg|\mathbb{E}f\bigg(\frac{S_{\nu_{n}}}{B_{\nu_{n}}}\bigg)- \mathbb{E}f(X^{*})\bigg| \leq M_{1} \mathbb{E}\bigg\{\frac{1}{B_{\nu_{n}}}\omega(f^{\prime}; B^{-1}_{\nu_{n}})  \bigg\}.
 \end{equation}
 If  in addition to the hypotheses $f\in Lip (\alpha, K)$ for $\alpha\in (0, 1]$ and $K\geq 0,$ then
\[
 \bigg|\mathbb{E}f\bigg(\frac{S_{\nu_{n}}}{B_{\nu_{n}}}\bigg) - \mathbb{E}f(X^{*})\bigg| =\mathcal{O}\bigg\{\mathbb{E}\bigg(\frac{1}{B^{1+\alpha}_{\nu_{n}}}\bigg) \bigg\}
 \] 
 as $n\to\infty.$
 \end{theorem}
\begin{proof}
Following Lemma \ref{lem:2.1} and according to \cite{Gut2005}, for every $f\in C^{1}_{B}(\mathbb{R}),$ we have
\begin{equation}\label{equ:3.6}
\begin{split}
&\bigg|\mathbb{E}f\bigg(\frac{S_{\nu_{n}}}{B_{\nu_{n}}}\bigg)- \mathbb{E}f(X^{*})\bigg| =
\sum\limits_{n=1}^{\nu_{n}}P(\nu_{n}=n)\bigg|\mathbb{E}f\bigg(\frac{S_{n}}{B_{n}}\bigg)- \mathbb{E}f(X^{*})\bigg| \\
&\leq \sum\limits_{n=1}^{\nu_{n}}P(\nu_{n}=n)\sum\limits_{j=1}^{n}\bigg|\mathbb{E}f\bigg(\frac{X_{j}}{B_{n}}\bigg)- f\bigg(\frac{X^{*}_{j}}{B_{n}}\bigg)\bigg|,
\end{split}
\end{equation}
 Since $f\in C^{1}_{B}(\mathbb{R}),$ one has by the Taylor series expansion
\[
f\bigg(\frac{X_{j}}{B_{n}}\bigg)=f(0)+\frac{f^{\prime}(0)}{B_{n}}X_{j}+\frac{1}{B_{n}}\bigg[f^{\prime}(\eta_{1})-f^{\prime}(0)\bigg]X_{j},
\] 
where $|\eta_{1}|\leq B^{-1}_{n}|x|.$ Applying the expectation to $f,$ since $\mathbb{E}X_{j}=0$ for $j=1, 2, \cdots$ and using the properties of the modulus of function $f,$ this yields
 \begin{equation}\label{equ:3.7}
\begin{split}
&\bigg|\mathbb{E}f\bigg(\frac{X_{j}}{B_{n}}\bigg)-f(0)\bigg|\leq \frac{1}{B_{n}}\int\limits_{\mathbb{R}}\bigg|f^{\prime}(\eta_{1})-f^{\prime}(0) \bigg| |x| dF_{j}(x)\\
&\leq  \frac{1}{B_{n}}\int\limits_{\mathbb{R}}\omega (f^{\prime}; B^{-1}_{n}|x|) |x| dF_{j}(x)\leq  \frac{1}{B_{n}}\omega (f^{\prime}; B^{-1}_{n})\int\limits_{\mathbb{R}}|x|(1+|x|) dF_{j}(x)\\
&\leq  \frac{1}{B_{n}}\omega (f^{\prime}; B^{-1}_{n})\bigg[\int\limits_{\mathbb{R}}|x|dF_{j}(x)+\int\limits_{\mathbb{R}}x^{2}dF_{j}(x)\bigg]\leq  \frac{1}{B_{n}}\omega (f^{\prime}; B^{-1}_{n})\bigg[\mathbb{E}|X_{j}|+\sigma^{2}_{j}\bigg].
\end{split}
\end{equation}
Analogously, since $f\in C^{1}_{B}(\mathbb{R})$ and $X^{*}_{1}, X^{*}_{2}, \cdots$ are independent normal distributed random variables with zero means and positive finite variances $\sigma^{2}_{j}\in (0, +\infty),$ we conclude  that
\begin{equation}\label{equ:3.8}
\begin{split}
&\bigg|\mathbb{E}f\bigg(\frac{X^{*}_{j}}{B_{n}}\bigg)-f(0)\bigg|\leq \frac{1}{B_{n}}\int\limits_{\mathbb{R}}\bigg|f^{\prime}(\eta_{2})-f^{\prime}(0) \bigg| |x| d\Phi_{j}(x)\\
&\leq  \frac{1}{B_{n}}\int\limits_{\mathbb{R}}\omega (f^{\prime}; B^{-1}_{n}|x|) |x| d\Phi_{j}(x)\leq  \frac{1}{B_{n}}\omega (f^{\prime}; B^{-1}_{n})\int\limits_{\mathbb{R}}|x|(1+|x|) d\Phi_{j}(x)\\
&\leq  \frac{1}{B_{n}}\omega (f^{\prime}; B^{-1}_{n})\bigg[\int\limits_{\mathbb{R}}|x|d\Phi_{j}(x)+\int\limits_{\mathbb{R}}x^{2}d\Phi_{j}(x)\bigg]\leq  \frac{1}{B_{n}}\omega (f^{\prime}; B^{-1}_{n})\bigg[\mathbb{E}|X^{*}_{j}|+\sigma^{2}_{j}\bigg].
\end{split}
\end{equation}
where $|\eta_{2}|\leq B^{-1}_{n}|x|.$\\
Combining the estimates (\ref{equ:3.7}) and (\ref{equ:3.8}), from (\ref{equ:3.4}) it may be concluded that
\begin{equation}\label{equ:3.9}
\begin{split}
\sum\limits_{j=1}^{n}\bigg|\mathbb{E}f\bigg(\frac{X_{j}}{B_{n}}\bigg)-\mathbb{E}f\bigg(\frac{X^{*}_{j}}{B_{n}}\bigg) \bigg|&\leq  \frac{1}{B_{n}}\omega (f^{\prime}; B^{-1}_{n})\sum\limits_{j=1}^{n}\bigg[\mathbb{E}|X_{j}|+\mathbb{E}|X^{*}_{j}|+\sigma^{2}_{j}\bigg]\\
&\leq  \frac{M_{1}}{B_{n}}\omega (f^{\prime}; B^{-1}_{n}).
\end{split}
\end{equation}
Applying  (\ref{equ:3.6}), we can assert that
\begin{equation}\label{equ:3.10}
\begin{split}
\bigg|\mathbb{E}f\bigg(\frac{S_{\nu_{n}}}{B_{\nu_{n}}}\bigg)- \mathbb{E}f(X^{*})\bigg|&\leq \sum\limits_{n=1}^{\nu_{n}}P(\nu_{n}=n)\sum\limits_{j=1}^{n}\bigg|\mathbb{E}f\bigg(\frac{X_{j}}{B_{n}}\bigg)- f\bigg(\frac{X^{*}_{j}}{B_{n}}\bigg)\bigg|\\
&\leq M_{1} \mathbb{E}\bigg\{\frac{1}{B_{\nu_{n}}}\omega(f^{\prime}; B^{-1}_{\nu_{n}})  \bigg\}. 
\end{split}
\end{equation}
 Finally,  if  in addition to the hypotheses $f\in Lip (\alpha, K), \alpha\in (0, 1],$ that is, 
 \[
 \omega(f^{\prime}; B^{-1}_{n})\leq KB^{-\alpha}_{n}\quad\text{for}\quad \alpha\in (0, 1]\quad\text{and}\quad K\geq 0.
 \]
 From (\ref{equ:3.10}),  for $\alpha\in (0, 1],$ we conclude that
\[
 \bigg|\mathbb{E}f\bigg(\frac{S_{\nu_{n}}}{B_{\nu_{n}}}\bigg) - \mathbb{E}f(X^{*})\bigg| =\mathcal{O}\bigg\{\mathbb{E}\bigg(\frac{1}{B^{1+\alpha}_{\nu_{n}}}\bigg) \bigg\}\quad\text{as}\quad n\to\infty.
 \] 
The proof is complete.
\end{proof}
 \begin{remark}\label{rem:3.3}\quad When $P(\nu_{n}=n)=1,$ from Theorem \ref{the:3.2},  we have the order of approximation estimate for the non-random Rotar's theorem \ref{the:1.4} as follows:
 \begin{equation*}
  \bigg|\mathbb{E}f\bigg(\frac{S_{n}}{B_{n}}\bigg)- \mathbb{E}f(X^{*})\bigg| \leq  \frac{M_{2}}{B_{n}}\omega(f^{\prime}; B^{-1}_{n})
 \end{equation*}
 for any $f\in C^{1}_{B}(\mathbb{R}).$\\
 Moreover,  if  $f\in Lip (\alpha, K), \alpha\in (0, 1],$ then
\[
 \bigg|\mathbb{E}f\bigg(\frac{S_{n}}{B_{n}}\bigg) - \mathbb{E}f(X^{*})\bigg| =\mathcal{O}\bigg\{\frac{1}{B^{1+\alpha}_{n}}\bigg\}
 \] 
 as $n\to\infty.$
 \end{remark}
  \begin{remark}\label{rem:3.4}\quad An estimate of order of approximation in the Rotar CLT for non-random sums of independent Bernoulli random variables is established by Formanov \cite{Formanov2002} as follows:
  \[
  \sup\limits_{t\leq T}\sum\limits_{j=1}^{n}|\Delta(g_{j}(t))|=\mathcal{O}\bigg(\frac{1}{B_{n}} \bigg)\quad\text{as}\quad n\to\infty,
  \]
  where $g_{j}(t)=\mathbb{E}e^{iY_{j}t}$ denotes the characteristic function of a Bernoulli distributed random variable  
  \[
  Y_{j}=
  \begin{cases}
  1,&\quad\text{with probability}\quad p_{j},\\
  0, &\quad\text{with probability}\quad 1-p_{j}, 
  \end{cases}
  \]
where $p_{j}\in (0,1), j=1, 2, \cdots$  and $\Delta(g_{j}(t))$ stands for the Stein-Tikhomirov operator connected with the characteristic function $g_{j}(t)$ (see for instance \cite{Formanov2002}, page 609). \\
  It is worth noting that,  for $T>0,$ the condition  
  \[
  \lim\limits_{n\longrightarrow\infty}\sup\limits_{t\leq T}\sum\limits_{j=1}^{n}|\Delta(g_{j}(t))|= 0
  \]
   is sufficient (and necessary) for validity of the Rotar CLT (see \cite{Formanov2002}, Theorem on page 607, for more details).
 \end{remark}
 \begin{theorem}\label{the:3.3}("Small-o" approximation estimates)\quad Under the given assumptions for the sequences $(X_{j}, j\geq 1), (X^{*}_{j}, j\geq 1),$  and $(\nu_{n}, n\geq 1)$ for $j=1, 2, \cdots; n\geq 1,$ let us suppose that  
 \begin{equation}\label{equ:3.11}
  \sum\limits_{j=1}^{n}\bigg[\mathbb{E}|X_{j}|+\mathbb{E}|X^{*}_{j}|\bigg]\leq M_{2}<+\infty,
  \end{equation}
 for a positive constant  $M_{2}>0.$  Furthermore,  assume that the random Rotar condition (\ref{equ:2.8}) is valid for every $\epsilon > 0.$ Then, for any $f\in C^{1}_{B}(\mathbb{R})$
 \[
 \bigg|\mathbb{E}f\bigg(\frac{S_{\nu_{n}}}{B_{\nu_{n}}}\bigg) - \mathbb{E}f(X^{*})\bigg| =o\bigg\{\mathbb{E}\bigg(\frac{1}{B_{\nu_{n}}}\bigg)\bigg\}\quad\text{as}\quad n\to\infty.
 \] 
 \end{theorem}
 \begin{proof}
For $f\in C^{1}_{B}(\mathbb{R}),$  we have
 \begin{equation*}
 \begin{split}
  &\mathbb{E}f\bigg(\frac{X_{j}}{B_{n}}\bigg) - \mathbb{E}f\bigg(\frac{X^{*}_{j}}{B_{n}}\bigg)=\int\limits_{-\infty}^{+\infty}f\bigg(\frac{x}{B_{n}}\bigg)d\bigg(F_{j}(x)-\Phi_{j}(x)\bigg)\\ 
  &=\int\limits_{-\infty}^{+\infty}\bigg[f\bigg(\frac{x}{B_{n}}\bigg)-\frac{f^{\prime}(0) x}{B_{n}}+\frac{f^{\prime}(0) x^{2}}{B^{2}_{n}} \bigg] d\bigg(F_{j}(x)-\Phi_{j}(x)\bigg),
 \end{split}
\end{equation*}
where we have used the fact that
    \[
    \int\limits_{-\infty}^{+\infty}x^{j}dF_{j}(x)=\int\limits_{-\infty}^{+\infty}x^{j}d\Phi_{j}(x)\quad\text{for}\quad j=1, 2.
    \]
   Applying the formula for integration by part to the integral
    \[
    \int\limits_{-\infty}^{+\infty}\bigg[f\bigg(\frac{X_{j}}{B_{\nu_{n}}}\bigg)-\frac{f^{\prime}(0) x}{B_{n}}+\frac{f^{\prime}(0) x^{2}}{B^{2}_{n}} \bigg] d\bigg(F_{j}(x)-\Phi_{j}(x)\bigg),
    \]
with the fact that  
\[
\lim\limits_{x\to\infty}x^{2}[1-F_{j}(x)+F_{j}(-x)]= 0,
\]
 and 
 \[
 \lim\limits_{x\to\infty} x^{2}[1-\Phi_{j}(x)+\Phi_{j}(-x)]=0. 
 \]
 Then, for every $\epsilon>0,$ we obtain
 \begin{equation}\label{equ:3.12}
 \begin{split}
  &\bigg|\mathbb{E}f\bigg(\frac{X_{j}}{B_{\nu_{n}}}\bigg) - \mathbb{E}f\bigg(\frac{X^{*}_{j}}{B_{\nu_{n}}}\bigg)\bigg|=\bigg|\int\limits_{-\infty}^{+\infty}\bigg[f\bigg(\frac{X_{j}}{B_{\nu_{n}}}\bigg)-\frac{f^{\prime}(0) x}{B_{n}}+\frac{f^{\prime}(0) x^{2}}{B^{2}_{n}} \bigg] d\bigg(F_{j}(x)-\Phi_{j}(x)\bigg)\bigg|\\ 
  &=\bigg|\int\limits_{-\infty}^{+\infty}\bigg[\frac{f^{\prime}(\frac{x}{B_{n}})}{B_{n}}-\frac{f^{\prime}(0)}{B_{n}}+\frac{f^{\prime}(0) x}{B^{2}_{n}} \bigg] \cdot \bigg(F_{j}(x)-\Phi_{j}(x)\bigg)dx\bigg|\\
  &\leq \frac{||f^{\prime}||}{B^{2}_{n}} \int\limits_{|x|\leq \epsilon B_{n}}|x|\cdot |F_{j}(x)-\Phi_{j}(x)|dx+\frac{||f^{\prime}||}{B^{2}_{n}} \int\limits_{|x|> \epsilon B_{n}}|x|\cdot |F_{j}(x)-\Phi_{j}(x)|dx \\
  &\leq \frac{||f^{\prime}||}{B_{n}}\epsilon\bigg[\mathbb{E}|X_{j}|+\mathbb{E}|X^{*}_{j}|\bigg]+\frac{||f^{\prime}||}{B^{2}_{n}} \int\limits_{|x|> \epsilon B_{n}}|x|\cdot |F_{j}(x)-\Phi_{j}(x)|dx.
 \end{split}
\end{equation}   
Summing over the j's both sides of the (\ref{equ:3.12}), on account of  (\ref{equ:3.11}), one has
    \begin{equation}\label{equ:3.13}
    \begin{split}
   & \sum\limits_{j=1}^{n}\bigg|\mathbb{E}f\bigg(\frac{X_{j}}{B_{n}}\bigg) - \mathbb{E}f\bigg(\frac{X^{*}_{j}}{B_{n}}\bigg)\bigg|\leq 
    \frac{||f^{\prime}||}{B_{n}}\epsilon\sum\limits_{j=1}^{n}\bigg[\mathbb{E}|X_{j}|+\mathbb{E}|X^{*}_{j}|\bigg]\\
    &+\frac{||f^{\prime}||}{B^{2}_{n}}\sum\limits_{j=1}^{n} \int\limits_{|x|> \epsilon B_{n}}|x|\cdot |F_{j}(x)-\Phi_{j}(x)|dx\leq \frac{||f^{\prime}|| M_{2}}{B_{n}}\epsilon+\frac{||f^{\prime}||}{B^{2}_{n}}\sum\limits_{j=1}^{n} \int\limits_{|x|> \epsilon B_{n}}|x|\cdot |F_{j}(x)-\Phi_{j}(x)|dx. 
    \end{split}
    \end{equation}  
Following (\ref{equ:3.11}), there exists a natural $n_{n}\geq 1,$ such that for $n\geq n_{0}$ we have
\[
M_{2}\geq 1.
\] 
Furthermore, for $f\in C^{\prime}_{B}(\mathbb{R}),$ we can choose the function $f$ such that $||f^{\prime}||\geq 1.$  
Then, used the fact that
\[
\bigg(\frac{B_{n}}{||f^{\prime}||M_{2}}\bigg)\leq B_{n}.
\]
Thus, from inequality (\ref{equ:3.13}), we see that
     \begin{equation}\label{equ:3.14}
     \begin{split}
&\sum\limits_{j=1}^{n}\bigg|\mathbb{E}f\bigg(\frac{X_{j}}{B_{n}}\bigg) - \mathbb{E}f\bigg(\frac{X^{*}_{j}}{B_{n}}\bigg)\bigg|\times\bigg(\frac{B_{n}}{||f^{\prime}||M_{2}}\bigg)\leq \sum\limits_{j=1}^{n}\bigg|\mathbb{E}f\bigg(\frac{X_{j}}{B_{n}}\bigg) - \mathbb{E}f\bigg(\frac{X^{*}_{j}}{B_{n}}\bigg)\bigg|\times B_{n}\\
&\leq \epsilon + \frac{1}{B_{n}}\sum\limits_{j=1}^{n} \int\limits_{|x|> \epsilon B_{n}}|x|\cdot |F_{j}(x)-\Phi_{j}(x)|dx. 
\end{split}
    \end{equation}    
Following \cite{Gut2005}, from (\ref{equ:3.14}), we conclude that  
  \begin{equation}\label{equ:3.15}
    \begin{split}
  & \bigg|\mathbb{E}f\bigg(\frac{S_{\nu_{n}}}{B_{\nu_{n}}}\bigg) - \mathbb{E}f(X^{*})\bigg|\times B_{\nu_{n}}\leq\mathbb{E}\bigg\{ \sum\limits_{j=1}^{\nu_{n}}\bigg|\mathbb{E}f\bigg(\frac{X_{j}}{B_{\nu_{n}}}\bigg) - \mathbb{E}f\bigg(\frac{X^{*}_{j}}{B_{\nu_{n}}}\bigg)\bigg|\times B_{\nu_{n}}\bigg\}\\
 & \leq \epsilon + \mathbb{E}\bigg\{\frac{1}{B_{\nu_{n}}}\sum\limits_{j=1}^{\nu_{n}} \int\limits_{|x|> \epsilon B_{\nu_{n}}}|x|\cdot |F_{j}(x)-\Phi_{j}(x)|dx\bigg\}. 
    \end{split}
    \end{equation}  
Applying the random Rotar condition (\ref{equ:2.7}), by virtue of arbitrariness of  $\epsilon,$  from (\ref{equ:3.15}) we obtain the final estimate as $n\to\infty$
\[
 \bigg|\mathbb{E}f\bigg(\frac{S_{\nu_{n}}}{B_{\nu_{n}}}\bigg) - \mathbb{E}f(X^{*})\bigg| =o\bigg\{\mathbb{E}\bigg(\frac{1}{B_{\nu_{n}}}\bigg)\bigg\}.
 \]  
The proof of theorem is complete. 
    \end{proof} 
\begin{remark}\label{rem:3.5}\quad Theorem  \ref{the:3.3} also reconfirms  that the random Rotar condition (\ref{equ:2.8}) is sufficient condition of the random Rotar CLT.
     \end{remark}
\begin{remark}\label{rem:3.6}\quad If $P(\nu_{n}=n)=1,$ following Theorem \ref{the:3.3}, for any $f\in C^{1}_{B}(\mathbb{R}),$ we obtain
 \[
 \bigg|\mathbb{E}f\bigg(\frac{S_{n}}{B_{n}}\bigg) - \mathbb{E}f(X^{*})\bigg| =o\bigg(\frac{1}{B_{n}}\bigg)\quad\text{as}\quad n\to\infty
 \] 
 for the order of approximation estimation  in  Rotar's theorem \ref{the:1.4}.
 \end{remark}

\noindent
Author's address:\\
Tran Loc Hung\\
Ho Chi Minh City, Vietnam.\\
tlhungvn@gmail.com

\end{document}